\documentclass[12pt,a4paper,twoside,leqno]{amsart}

\usepackage[latin1]{inputenc}

\usepackage{amsmath}
\usepackage{amsfonts}
\usepackage{amssymb}
\usepackage{amstext}
\usepackage{amsbsy}
\usepackage{color}

\newtheorem{corollary}{\sc Corollary}[section]
\newtheorem{theorem}{\sc Theorem}[section]
\newtheorem{lemma}{\sc Lemma}[section]
\newtheorem{proposition}{\sc Proposition}[section]
\newtheorem{remark}{\sc Remark}[section]
\newtheorem{definition}{\sc Definition}[section]

\newcommand{\thmref}[1]{Theorem~\ref{#1}}
\newcommand{\secref}[1]{Section~\ref{#1}}
\newcommand{\lemref}[1]{Lemma~\ref{#1}}
\newcommand{\remref}[1]{Remark~\ref{#1}}

\newcommand{\defref}[1]{Definition~\ref{#1}}

\def\Hip{\mathbb H}

\let\hip=\Hip
\let\hyp=\Hip

\newcommand{\R}{\mathbb R}

\let\h=\hip

\let\re=\R

\DeclareMathOperator{\HMc}{\mathcal {M_H}}
\DeclareMathOperator{\bMc}{\mathcal {\overl{M}_H}}
\DeclareMathOperator{\Lc}{\mathcal {L}}
\DeclareMathOperator{\osc}{osc}

\def\Fc{\mathcal F}

\let\Pc =\P

\def\Rc{\mathcal R}

\def\Nc{\mathcal N}

\def\Om{\Omega}

\def \la{\lambda}
\def \a{\alpha}

\def \ga{\gamma}

\def \de{\delta}

\def \a{\alpha}

\def \ga{\gamma}
\def \Ga{\Gamma}

\def \de{\delta}
\def \e{\varepsilon}

\def \ep{\epsilon}
\def \la{\lambda}
\def \La{\Lambda}

\def\vphi{\varphi}

\def \Om{\Omega}

\def \dd   {\displaystyle}

\def\leqs{\leqslant}
\def\geqs{\geqslant}

\def\rmd{\mathop{\rm d\kern -1pt}\nolimits}
\def\rme{\mathop{\rm e\kern -1pt}\nolimits}

\def\e{\varepsilon}

\DeclareMathOperator{\diver}{div}
 \DeclareMathOperator{\cst}{cst}

\DeclareMathOperator{\diam}{diam}

\def \noi {\noindent}

\def\bel{ \medskip
 \centerline{$ \ast \hbox to 1.0cm{}\ast \hbox to 1.0cm{}\ast $}
}

\def\wdh{\widehat}

\def\overl{\overline}

\def\goto{\rightarrow}

\def\longerrightarrow{-\kern-5pt\longrightarrow}

\def\vphi{\varphi}
\def\star{\lower 1pt\hbox{*}}
\def \nulset {
\raise 1pt\hbox{ \hskip -3pt$\not$\kern -0.2pt \raise
.7pt\hbox{${\scriptstyle\bigcirc}$}}}

\def \ep{\epsilon}

\newcommand{\hi}[1]{\mathbb{H}^#1}

\let\leq=\leqslant

\begin{document}

\title[minimal]
{\small  Uniform a priori estimates   for a   class of   horizontal minimal equations.}

\author[  Ricardo Sa Earp]
{\scshape Ricardo Sa Earp}

 \address{Departamento de Matem\'atica,
  Pontif\'\i cia Universidade Cat\'olica do Rio de Janeiro, Rio de Janeiro,  22453-900 RJ,
 Brazil }\email{earp@mat.puc-rio.br}

\thanks{The  author wish to thank
CNPq, FAPERJ of Brazil, for financial support}

\date{\today}


% \keywords{barrier}

\bigskip

\begin{abstract}
\noi  In the product space $\hip^n\times\R,$  we obtain uniform a priori  $C^0$ horizontal length estimates, uniform
 a priori  $C^1$  boundary gradient estimates, as well as uniform modulus of continuity,  for  a class of  horizontal  minimal equations. In two independent variables,
we derive a certain uniform global  a priori $C^1$ estimates and we  infer
an exis\-tence  result.
\end{abstract}

\maketitle

%\textbf{Mathematics Subject Classification (2000):~} 53C42, 35J25

\medskip

% \vspace{2cm}

\section{Introduction }\label{intro}

The theory of minimal and constant mean curvature surfaces in the product  space $\hip^2 \times \R,$  where $\hip^2$ is the
 hyperbolic  plane  is now  a rich field of intense research.
The notion of vertical graph has a major importance in this theory. The first main results about minimal vertical graphs were derived by B. Nelli and H. Rosenberg \cite{NR}. Since then,
 a significant progress in the theory  has been achieved, see, for instance,  \cite{SE-T}, \cite{M-R-R}, \cite{M-R-R2}. When the ambient space is  $\hip^n \times \R,$ the notion of  vertical mean curvature
 equation in $n$ independent variables has also been established and developed \cite{Sp}, \cite{SE-T4}.

 On the other hand, there exists also  a notion of graph called {\em horizontal graph}
that has been focused  in the theory of  minimal and constant mean
curvature surfaces  \cite{Sa}, \cite{Sa2}.

The notion of horizontal graph arises naturally
in the theory of hypersurfaces in $\hip^n \times \R$. There are many interesting examples of
minimal and constant mean curvature horizontal graphs in $\hip^n
\times \R$ given by  explicit formulas. In fact, there
are complete horizontal  minimal graphs and there are entire horizontal
constant mean curvature graphs  \cite{Be-SE1}, \cite{Be-SE2}, \cite{E-SE}.

We notice that there is a different notion of horizontal graph called {\em horizontal graph with respect to a geodesic of $\hip^2$}, that has been applied to prove  a Schoen type result for minimal  surfaces in $\hip^2 \times \R$  \cite{H-N-SaE-T}.

We
choose the upper half-plane model of hyperbolic plane
$\hip^2=\{(x, y), y>0\},$ endowed with the hyperbolic metric $\dd
d\sigma^2=\frac{dx^2 + dy^2}{y^2}.$ A horizontal graph  in $\hyp^2
\times \R$  is the set $S=\{(x, g(x, t), t),\, (x, t)\in
\Om\}\subset \hyp^2 \times \R\ $, where $\Om\subset\partial_\infty
\Hip^2\times\R$ is a domain and $g(x, t)>0,$ for every $(x, t)\in
\Om.$ This means that in any slice of  $\hyp^2 \times \R$ given by
$t=\cst,$ each horizontal geodesic $x=\cst, y>0$ intersects $S$ in
one point at most. We call the positive function $g(x, t)$  the
{\em horizontal length} of the graph. If $S$ is a horizontal
minimal graph in $\hyp^2 \times \R$ the positive function $g(x,
t)$ satisfies equation \eqref{2min}.

\vskip2mm

 In this paper, we  derive uniform a priori   horizontal length estimates  and uniform a priori boundary gradient estimates  for
positive smooth solutions of a wider class of quasilinear
elliptic
 equations, indexed by the parameter $ \ep\in [0, 1]$. We call such  equations the $\ep$-{\em horizontal minimal equations}
 (see equation \eqref{ehnme}).
 If $\ep=0,$ we find the horizontal minimal equation.

    We point out that if $\ep>0,$ the equation is strictly elliptic, but if $\ep=0$ this fact  is no longer true in  general. Indeed, there are many  examples that can be constructed to show lack of strictly ellipticity for the horizontal minimal equation, even on bounded domains.  We give now a significant example of this phenomenon in two variables.   Just take $g(x,t)=x \sinh t,\, 0<x<1,\, 0<t<1$ \cite[Equation 32]{Sa}. It is easy to verify that $y=g(x,t)$ satisfies the horizontal minimal equation. A simple verification shows that if $x \goto 0, $ then the first eigenvalue of the horizontal minimal operator goes to zero.
      On the other hand, it is amazing  that the vertical minimal equation is strictly elliptic for all values of the independent variables \cite[Equation 4]{Sa}, \cite[Equation 6]{SE-T}.

 The uniform $C^0$ estimates for positive smooth solutions on bounded domains, continuous up to the boundary,  is obtained by comparing, by maximum principle, such solutions with certain  geometric subsolutions and supersolutions. These a priori  $C^0$ estimates depends on the width of the domain and  on the boundary value data.

   In the same sprit as in
\cite{asi},  we obtain uniform a priori boundary gradient estimates on
arbitrary bounded  smooth convex domains  for solutions of equation \eqref{ehnme}, that are smooth and positive up to the boundary.
   Indeed, we provide uniform analytic {\em barriers} on  bounded $C^0$
convex domains to ensure uniform modulus of continuity for
positive continuous solutions of   \eqref{ehnme} that
are continuous and positive up to the boundary.

In the case of two  variables,
we are able to derive  uniform a priori global  $C^1$ estimates on
 bounded smooth convex domains, making an
 additional strong assumption on the horizontal length. However,  this assumption is compatible with the geometry of the ambient space, {\em i.e.}, it is invariant by hyperbolic translations.

Notice that there is a non-existence result that follows from  the asymptotic principle proved by E. Toubiana and the author in \cite[Theorem 2.1]{SE-T}. Namely, there is
no horizontal minimal graph given by a function $g\in C^2(\Om)\cap
C^0(\overl{\Om})$ on a bounded strictly convex domain $\Om$,
taking zero  boundary data on $\partial \Om.$ Furthermore, there is no  horizontal minimal graph in $\hip^2 \times \R$, over a bounded
Jordan domain $\Om$ strictly contained in an horizontal slab of
height $\pi,$ given by a function $g\in C^2(\Om)\cap
C^0(\overl{\Om}),$ taking zero asymptotic boundary data on
$\partial \Om$ \cite[Corollary 2.1]{SE-T}.

We obtain an existence result for the $\ep$-horizontal
minimal equation in two independent variables  on  bounded smooth convex domains, taking certain
smooth positive boundary value data. This existence result is, in certain sense, a counterpart of the non-existence results stated in the foregoing paragraph. Particularly, we derive the following consequence. Given a smooth positive function $f$ on $\partial \Om$; if $c_0$ is a constant large enough,  then there exists a solution of the $\ep$- horizontal minimal equation in $\Om$, taking boundary value data $f+c_0$ on $\partial\Om$.  Actually, it suffices to take  $c_0$  greater than  the sum of the oscillation of the boundary data $f$ with the half of the horizontal width of the domain $\Omega.$

 Existence results for the Dirichlet problem for the $\ep$-horizontal minimal equation in $n\geqs 3$ variables is an open problem.

  It is worthwhile to notice that the horizontal minimal equation \eqref{2min} is not invariant by Euclidean translations along the {\em  $y$-axis}. Of course, the structure of the horizontal minimal equation does not ensure the uniqueness of the solution of the Dirichlet problem on bounded domains. In fact, the model minimal surfaces described in \cite[Proposition 2.1]{SE-T} and used in \cite[Proof of Theorem 1.1]{Sa2} to prove a Bernstein type theorem, shows that there is a  family of horizontal minimal graphs over a rectangle $\Rc$  of $t$-height greater than $\pi$   taking zero boundary data over $\partial \Rc.$
   Furthermore,  over domains of arbitrarily large $x$- width, each element of this family takes arbitrarily small constant boundary data over a  bounded domain in the rectangle, but the family  attains arbitrarily large horizontal length. Thus, it follows that   the dependence of the a priori horizontal length estimates  on the width of the domain  is  quite natural.

   However,   over certain admissible convex domains such that the boundary data has an admissible bounded slope condition,  the solution  is the Morrey´s solution of the Plateau problem \cite{Mo1}. This result follows from the uniqueness theorem established in \cite{Sa2}.  The uniqueness or the non-uniqueness of the Dirichlet problem for  the $\ep$- horizontal minimal equation on  bounded convex domains and positive boundary data is an open problem.

We remark that we need to
 consider the family of $\ep$- horizontal minimal equations in order to prove the existence of our Dirichlet problem for  the horizontal minimal equation.
  The scheme of our construction is the following. First, we prove the existence result in the strictly elliptic situation,{\em i. e.}  when $\ep>0.$
 Then, we are able to deduce it
 for $\ep=0;$ that is, we get a solution for the horizontal minimal equation.
 In fact, in view of our uniform a priori $C^1$ global estimates and elliptic  theory, this solution is obtained as the limit in the $C^2$-topology of a sequence $g_{\ep_n},$
   as $0<\ep_n\goto 0$, satisfying  the $\ep_n$- horizontal minimal equation.

\vskip1.5mm
{\sc Acknowledgments}: The author warmly thanks  to   Eric Toubiana and  Barbara Nelli for their valuable
observations.

\vskip2mm

\section{\sc  The $\ep$-horizontal minimal equation in ~ $\hyp^n \times \R$.}

In this section we  give some computations, some model
subsolutions and supersolutions and we write down some basic
properties of the $\ep$-horizontal minimal equation in the product
space $\hyp^n \times \R$.

We choose the upper half-plane model of the hyperbolic plane
$\hip^n=\{(x_1,x_2,\ldots,x_{n-1}, y), y>0\},$ endowed with the
hyperbolic metric $\dd d\sigma^2=\frac{dx_1^2 +\cdots + dx_{n-1}^2
+ dy^2}{y^2}.$ A horizontal graph  in $\hyp^n \times \R$  is the
set $S=\{(x_1,x_2,\ldots,x_{n-1}, g(x_1,x_2,\ldots,x_{n-1}, t),
t),\, (x_1,x_2,\ldots,x_{n-1}, t)\in \Om\}\subset \hyp^n \times \R
$, where $\Om\subset\partial_\infty \Hip^n\times\R$ \cite{cas} is a domain
and $g(x_1,x_2,\ldots,x_{n-1}, t)>0, (x_1,x_2,\ldots,x_{n-1},
t)\in \Om.$ This means that at any slice of $\hyp^n \times \R$
given by $t=\cst,$ each horizontal geodesic
$(x_1,x_2,\ldots,x_{n-1})=\cst, y>0$ intersects $S$ at most on one
point.

\vskip2mm
 We  give now the computations of some important geometric
quantities of the horizontal graph $S.$

Let us define the quantity\\ $\dd W:=g^2\left( 1 +g_{x_1}^2
+\cdots
 + g_{x_{n-1}}^2\right) +g_t^2.$

\begin{itemize}

 \item The
coefficients of the first fundamental form are \vskip2mm
 $ g_{kk}=\frac{1
+g_{x_k}^2}{g^2},\, 1\leqs k<n,\,
g_{jk}=\frac{g_{x_j}g_{x_k}}{g^2}, 1\leqs j<k<n,\\
g_{nk}=\frac{g_{x_k} g_t}{g^2},1\leqs k<n,\, g_{nn}=\frac{g^2
+g_t^2}{g^2}.$

\item The inverse of the matrix $(g_{ij})$ is the matrix $g^{ij}$
given by

\vskip2mm
 \noi $g^{kk}= \left[ g^2 \left( 1 +g_{x_1}^2 +\cdots +\wdh{g_{x_k}^2}
 +\cdots g_{x_{n-1}}^2\right)+g_t^2\right] \cdot\,\frac{g^2}{W}\, ,\\ 1\leqs k<n.$\\
  $g^{kj}=-g^2 g_{x_j} g_{x_k} \cdot\,\frac{g^2}{W},\, 1\leqs j< k<n.$\\
$ g^{nk}=-g_{t} g_{x_k}\cdot\,\frac{g^2}{W},\, 1\leqs  k<n.$\\
$g^{nn}=\left( 1 +g_{x_1}^2 +\cdots +
g_{x_{n-1}}^2\right)\cdot\,\frac{g^2}{W}.$

 \end{itemize}

 Let us assume that our horizontal graph $S$ is oriented by the
unit normal $N$ given by
\begin{equation}\label{unn}
 N= \left(-g_{x_1}g^2,-g_{x_2}g^2,\ldots,
-g_{x_{n-1}} g^2, g^2, -g_t\right)\cdot \,\frac{1}{W^{1/2}}
\end{equation}

\begin{itemize}

\item  The coefficients of the second fundamental  form are given by

\vskip2mm

$b_{kk}=\left(\frac{1}{g} + g_{x_kx_k}
+\frac{g_{x_k}^2}{g}\right)\cdot\,\frac{1}{W^{1/2}}, 1\leqs k<n.$

 \vskip2mm
 $b_{kj}=\left(g_{x_kx_j} +\frac{g_{x_k} g_{x_j}}{g}\right)\cdot\,\frac{1}{W^{1/2}}, 1\leqs ~j<~k<~n.$

 \vskip2mm

 $b_{nk}= g_{tx_k}\cdot\,\frac{1}{W^{1/2}},\, 1\leqs k<n.$

 \vskip2mm

 $b_{nn}=\left(g_{tt} -\frac{g_t^2}{g}\right)\cdot\,\frac{1}{W^{1/2}}.$

\end{itemize}

\vskip2mm

\begin{proposition}\label{hnmeane}

Let $S$ be a horizontal graph oriented by the unit normal $N$
given by \eqref{unn}. Let $H=H(N)$ be the  mean curvature of $S.$

 Then the horizontal mean curvature equation is given by

\begin{multline}\label{hne}
\HMc(g)=\frac{n H}{g^2}\left(g^2\left( 1 +g_{x_1}^2 +\cdots
 + g_{x_{n-1}}^2\right) +g_t^2\right)^{3/2}\hfill\\
 \text{where}\hfill\\
\HMc(g):= \sum\limits_{k=1}^{n-1}g_{x_k x_k}\left[
g^2\left(1+g_{x_1}^2+\cdots+\widehat{g_{x_k}^2}+\cdots
g_{x_{n-1}}^2\right) +g_t^2\right] +
\left(1+\sum\limits_{k=1}^{n-1}g_{x_k}^2\right)g_{tt}\hfill\\
-2\sum\limits_{k=1}^{n-1} g_{x_k} g_t g_{x_k t}
-2g^2\sum\limits_{1\leqs j<k\leqs n-1} g_{x_j} g_{x_k} g_{x_jx_k}
+(n-1) g \left( 1 + \sum\limits_{k=1}^{n-1} g_{x_k}^2\right)
+(n-2)\frac{g_t^2}{g}\hfill\\
\end{multline}
\end{proposition}
 Now, if $S$ is  a horizontal minimal graph in $\hyp^n \times \R$, {\em i.e.} $H=0,$
 then the
positive function $y=g(x_1,x_2,\ldots,x_{n-1}, t)$ that we call
the {\em horizontal length} satisfies the equation
\begin{multline}\label{hnme}
 \HMc(g)=0\hfill\\
\end{multline}
If $n=2$ we recover the horizontal minimal equation in
$\hip^2\times \R$ \cite[equation (2)]{Sa}:

 \begin{equation}\label{2min}
  \HMc(g):=
g_{xx} (g^2 +g_t^2) + g_{tt} (1 +g_x^2) -2g_xg_tg_{xt} +g(1
+g_x^2)=0.
\end{equation}
From now on, we focus on a  $1$- parameter family of elliptic
equations including  equation \eqref{hnme}.
 Given a  constant $\ep\in [0, 1]$, we say that a  positive  $C^2$ function $g$ on a  domain $\Om$ is a
solution of the {\em $\ep$-horizontal minimal equation}, if it
satisfies the following quasilinear elliptic P.D.E:
\begin{multline}\label{ehnme}
 \HMc_\ep(g):= \sum\limits_{k=1}^{n-1}\left[
g^2\left(1+g_{x_1}^2+\cdots+\widehat{g_{x_k}^2}+\cdots
g_{x_{n-1}}^2\right) +g_t^2+\frac{\ep}{n-1}\right]g_{x_k
x_k}\hfill\\ +
\left(1+\sum\limits_{k=1}^{n-1}g_{x_k}^2\right)g_{tt}
-2\sum\limits_{k=1}^{n-1} g_{x_k} g_t g_{x_k t}
-2g^2\sum\limits_{1\leqs j<k\leqs n-1} g_{x_j} g_{x_k}
g_{x_jx_k}\hfill\\
 +(n-1) g \left( 1 + \sum\limits_{k=1}^{n-1}
g_{x_k}^2\right)
+(n-2)\frac{g_t^2}{g}\hfill\\
\hphantom{0000000000}=0\hfill\\
\end{multline}
 Setting $a_{kk}(g, D g):= \left[
g^2\left(1+g_{x_1}^2+\cdots+\widehat{g_{x_k}^2}+\cdots
g_{x_{n-1}}^2\right) +g_t^2+ \frac{\ep}{n-1}\right], k=1,\ldots
n-1,\quad a_{nn}(g, D g)=
\left(1+\sum\limits_{k=1}^{n-1}g_{x_k}^2\right)$ and  $a_{kn}(g, D
g)= -g_{x_k} g_t,k=1,\ldots n-1,\quad a_{jk}(g, D
g)=-g^2\sum\limits_{1\leqs j<k\leqs n-1} g_{x_j} g_{x_k},$ then
the  symmetric matrix $a_{ij}(g, D g),\, i,j=1,\ldots,n$ is
positive, and satisfies
\begin{multline}\label{ell2}
\sum\limits_{i,j}a_{ij}(g, D
g)\xi_i\xi_j\leqs\text{trace} (a_{ij}(g, D g))\sum\limits_{k=1}^n \xi_k^2= \hfill\\
\Bigl(1 + \ep + g^2(n-1) +( g^2(n-2) +1)(g_{x_1}^2+\cdots+\cdots
g_{x_{n-1}}^2 ) +(n-1)g_t^2\Bigr)|\xi|^2\hfill\\
\leqs\Bigl(2 +  g^2 (n-1) +\left(\max\{1,g^2\} (n-2)
+1\right)\left(g_{x_1}^2 +\cdots +g_{x_{n-1}}^2
+g_t^2\right)\Bigr)|\xi|^2
\end{multline}
\newpage
And
\begin{multline}\label{ell}
\sum\limits_{i,j}a_{ij}(g, D
g)\xi_i\xi_j=\sum\limits_{k=1}^{n-1}\xi_k^2 (g^2 +\frac{\ep}{n-1})
+\xi_n^2 +g^2\!\!\!\sum\limits_{1\leqs j<k\leqs n-1} \left(\xi_k
g_{x_j}-\xi_j
g_{x_k}\right)^2\hfill\\
\hfill\hphantom{0000000000000000000}+\sum\limits_{k=1}^{n-1}\left(\xi_k g_t-\xi_n g_{x_k}\right)^2\hfill\\
\hphantom{0000000000000000000}\geqs
\min\{1,g^2\}|\xi|^2,\hfill\\
\quad\text{where}\, \xi\in\R^n\setminus\{0\},|\xi|^2=\sum\limits_{k=1}^n \xi_k^2.\hfill\\
\hfill\\
\end{multline}
By invoking   inequality \eqref{ell}, we deduce that the
$\ep$-horizontal minimal equation \eqref{ehnme} is a second order
quasilinear elliptic equation. Of course, if $\ep>0,$ equation
\eqref{ehnme} is strictly elliptic. If $\ep=0$ and if $g$ is a
$C^2$ positive function on a bounded domain $\Om$, continuous up
to the boundary, then \eqref{ehnme} is a strictly elliptic
equation (cf. \secref{she}, inequality \eqref{hesn}).

\begin{remark}
{\em
It is worth  noticing that L. Hauswirth, H. Rosenberg and
J. Spruck proved a half-space theorem   for properly embedded
constant constant mean curvature $1/2$ surfaces in $\hip^2\times
\R$, where the construction of $H=1/2$ horizontal graphs play a
significant role in the proof \cite[Theorem 1.1]{HRS}.}
\end{remark}

We recall now the definition of subsolution and supersolution
\cite{SE-T}, \cite{SE-T4}. We say that a $C^2$ function $u:\Om\goto \R$  is a
subsolution of  equation \eqref{ehnme} in a domain $\Om$, if
$\HMc_\ep(u)\geqs 0,$ in $\Om.$ We say that $w:\Om\goto \R$ is a
supersolution if $\HMc_\ep(w)\leqs 0.$  It is well-known that
elliptic theory ensures the following \cite{P-W},\cite{GT}, \cite{P-S}:

\begin{proposition}[{\bf {\em Classical maximum principle}} ] Let $u:\Om\goto \R$ and $w: \Om\goto \R$
be a subsolution and a supersolution,  for equation \eqref{ehnme},
respectively. Then if $u\leqs w$ in $\Om$ and if there is a point
$p\in \Om$ such that $u(p)=w(p),$ it follows that $u=w$ in $\Om.$
\end{proposition}

\begin{remark}\label{rm1}
${}$

{\em
\begin{enumerate}

\item Is is easy to see that   the Euclidean $n$-planes in $\hip^n\times \R$ given by \\ $y = \sum\limits_{k=1}^{n-1}a_kx_k +bt +c;\ a, b, c\in
\R, \, y>0$ are positive subsolutions  of  equation \eqref{ehnme}.
In particular the {\em horocylinders} given by $y = c,\, c>0,$ are
subsolutions of \eqref{ehnme}.

\item \label{geo}  Of course, the  {\em vertical geodesic $n$-planes} $\Pc=\Pc(R, x_0) =\\ \{(x_1,\ldots, x_{n-1}, y, t)\in \hip^
n\times\R;\\
y=\sqrt{R^2 -(x_1-a_1)^2-\cdots -(x_{n-1}-a_{n-1})^2},\,t\in
(-\infty, \infty),\\ R^2 -(x_1-a_1)^2-\cdots -(x_{n-1}-a_{n-1})^2>
0\},$ where $R>0$ and $x_1,\ldots, x_{n-1}\in \R$, are solutions
of equation \eqref{hnme}. Moreover, it is routine to check that
they are supersolutions for equation \eqref{ehnme}, as well.

\item  Minimality is invariant by a positive isometry of  $\hip^n\times \R$ given by a hyperbolic
translation of $\hip^n\times\{0\}$ along a geodesic $L$
\cite{cas}. Particularly,  the homothety in $\hip^n$ with center
$(a_1,\ldots, a_{n-1})\in \R^{n-1}$, ratio $\la >0,$  keeping
$\hip^n$ invariant gives rise to an isometry of the product space
$\hip^n\times \R$ given by $(x_1,\ldots, x_{n-1}, y, t)\mapsto
\left(\la \left[ (x_1,\ldots, x_{n-1}, y)- (a_1,\ldots, a_{n-1},
0)\right ] +(a_1,\ldots, a_{n-1}, 0),\, t\right).$

In view of this observation, we deduce  that if\\ $y
=g(x_1,\ldots, x_{n-1}, t)$ is a  positive solution of
\eqref{hnme} on a domain $\Om$ then $ y=\la
g(\frac{x_1}{\la},\ldots,\frac{x_{n-1}}{\la}, t),\, \la
>0$ is also a solution of \eqref{hnme} on $T_\la(\Om), $ where
$T_\la$ is the linear map given by the matrix $\begin{pmatrix}
\la\,I_{n-1}&0\\0&1\end{pmatrix}$, and where $I_{n-1}$ is the
$n-1\times n-1$ identity matrix.

\item  Important model minimal hypersurfaces are the
hypersurfaces invariant by hyperbolic translations \cite[Theorem
3.8]{Be-SE1}. They are complete horizontal minimal graphs.
%By
%employing the family, exactly as done in \cite{Sa2}, we infer that
%there is no entire horizontal minimal graph in $\hip^n\times \R.$
%Further applications are carried out in  a fourthcoming paper.
\end{enumerate}
}
\end{remark}

\section{\sc Uniform horizontal length estimates} \label{she}

In this section we infer uniform horizontal length estimates for
solutions of equation \eqref{ehnme} over bounded domains that are continuous and positive up to the boundary.  In order to do that, we find minimal and constant mean curvature surfaces that will be used   as  barriers.

We need to define some geometric quantities that will be useful in the sequel.

\begin{definition}\label{qua}

Let $\Om\subset\partial_\infty \Hip^n\times\R$  be a bounded $C^0$
domain with boundary $\Ga$. Let  $f\in C^ {0}(\Ga)$ be a positive function. If $n=2$, let $h(\Ga)= \max\limits_\Ga x|_\Ga-\min\limits_\Ga x|_\Ga$ be the horizontal width of the domain $\Om,$ where $ x|_\Ga$ is the  restriction of the  $x$ coordinate to $\Ga.$
 We define
$$R(\Om,f):=\sqrt{\max\limits_\Ga f^2 +\left(\frac{
h(\Ga)}{2}\right)^2}$$

If $n\geqs 3,$  let $\Om^\pi$ be the orthogonal
projection of $\Om$ on $\partial_\infty \Hip^n\times\{0\}$. We
denote by $\diam (\Om^\pi)$  the diameter of $\Om^\pi.$

We define  $r(\Om)$ by the radius of
the smallest $n-1$ closed round  disk containing
$\Om^\pi$, with center in $\Om^\pi$. Clearly,\\
$\frac{\diam(\Om^\pi)}{2}\leqs r(\Om)< \diam
(\Om^\pi).$

We set
$$R(\Om,f):=\sqrt{\max\limits_\Ga f^2
+r^2(\Om)}$$

\end{definition}

 We
have therefore:

\begin{lemma}[{\bf{\em Uniform horizontal length estimates}}]\label{hhen}
${}$

Let $\Om\subset\partial_\infty \Hip^n\times\R$ be a bounded $C^0$
domain with boundary $\Ga$. Let $f\in C^ {0}(\Ga)$ be a positive
function.

Assume that $f$ admits a positive  extension $g\in
C^{2}(\Omega)\cap C^0(\overl{\Om})$ satisfying the
$\ep$-horizontal minimal surface equation in $\Omega$.

Then, the following estimate holds:

\begin{equation}\label{hesn}
 \hfill\min\limits_\Ga f<g< R(\Om,f)\quad \text{in} \quad \Om.\hfill\\
\end{equation}

\end{lemma}

\begin{proof} If $n=2$ then the  proof is based on the observation
that the horocylinders and the vertical geodesic planes are
subsolutions and supersolutions for equation \eqref{ehnme},
respectively. Let us proceed now the proof. Let $S$ be the graph of $g$, let  $y|_S$ be the horizontal
coordinate $y$ restricted to $S$, given by the function $g$ and let $y|_{\partial S}$ be the horizontal
coordinate $y$ restricted to $\partial S$, given by the function $f$

First, we will deduce the lower bound.  Let $c_0$ be a positive
constant such that $c_0<\min\limits_S y|_S.$  Consider the family
of horocylinders  $P_s, s\geqs 0$ given by $y=c_0 +s, s\geqs 0.$
Of course, for $s=0$, $S$ is contained in the mean convex side of
$P_0$, that is, the horizontal length of $S$ satisfies $y|_S>
c_0.$ Now letting $s\uparrow \infty$ we find a first point of
contact of $S$ with the family $P_s.$ By the maximum principle
this first point of contact should be at a point of the boundary
$\partial S,$ hence  the whole surface $S$ should be strictly
contained in the mean convex side of the horocylinder
$y=\min\limits_\Gamma f,$ {\em i.e} $y|_{S\setminus \partial S}
> \min\limits_\Gamma f,$ or equivalently $g>\min\limits_\Ga f$ in $\Om$, as desired.

To obtain the upper bound, we argue as follows.\\
Consider the geodesic plane $\Pc(\Gamma)$ given by \remref{rm1}
(\ref{geo}), where $x_0=x_0(\Gamma):=\left(\max\limits_\Gamma
x|_{\Ga}
+\min\limits_\Gamma x|_{\Ga}\right)/2$ and $R=R(\Om,f)=\\
\sqrt{\max\limits_\Ga (y|_{\Ga})^2 +\left( \frac{\max\limits_\Ga
x|_{\Ga}}{2}- \frac{\min\limits_\Ga x|_{\Ga}}{2}\right)^2}.$ By
construction, $\partial S \subset\Pc^+(\Gamma):=\{(x, y, t)\in
\hip^2\times \R; \,y\leqs\sqrt{R^2(\Om,f) -(x-x_0(\Gamma))^2}\}.$
Consider now the family of geodesic planes $\Pc_s=\Pc(R(\Gamma)+s,
x_0(\Gamma)),\, s\geqs 0, \Pc_0=\Pc(\Gamma).$ Of course, for $s$
big enough, we have that $S$ is contained in the same connected
component of $\hip^2\times \R\setminus \Pc_s(\Gamma)$ as $\partial S,$
that is $S\subset \{(x, y, t)\in \hip^2\times \R;\,y
<\sqrt{(R(\Gamma)+s)^2 -(x-x_0(\Gamma))^2}\}.$ Now letting
$s\downarrow 0$ we cannot find a point of contact for $s>0,$ by
maximum principle. In contrary, we should find a first interior
point of contact of $S$ with some geodesic plane of the family for
some $s_0>0$. Since this first point of  contact should be an
interior point, we deduce that $S$ would be part of the geodesic
plane $\Pc_{s_0}$. Hence, the horizontal projection of whole $S$
in $\hip^2\times\{0\}$, would be contained in the geodesic (half
Euclidean circle) $y^2 +(x-x_0(\Gamma))^2 = (R(\Om,f) +s)^2,y>0,$
which gives a contradiction. We conclude therefore that $S\subset
\Pc^+(\Gamma)$. Thus, we have $y|_{S}\leqs R(\Om,f).$

For the same reason we cannot find an interior point of contact
when, during the movement toward the original position at $s=0$,
the family reaches $\Pc(\Gamma)$. Indeed if this could happen then
$S$ would be a part of the geodesic plane $\Pc(\Gamma)$, so that
the horizontal projection of $\Ga$ would be an arc of the geodesic
$y^2 +(x-x_0(\Gamma))^2 = (R(\Om,f))^2,y>0$ with positive length.
Since the intersection of the horizontal projection of $\Ga$ with
such geodesic consists at most of two points, we get a
contradiction. Henceforth, we obtain the strictly inequality
$y|_{S\setminus\Gamma} < R(\Om,f)$. This completes the proof of the
Lemma, if $n=2.$

Now,  if $n>2$ the structure of the proof is the same. Let us consider again $S=\text{
graph (g)}.$ Notice that our assumption implies that $\partial S$
lies in the side of a $n$-dimensional vertical geodesic plane
$S_{R(\Om,f)} \times \R\subset \hip^n\times R$, whose asymptotic
boundary contains $\Om^\pi$. Where $S_{R(\Om,f)}$ is a $n-1$
geodesic plane in $\Hip^n$ (Euclidean $n-1$ halfsphere) of
Euclidean radius $R(\Om,f).$

 We now accomplish the proof, as follows:
 We can use the $n$ dimensional horocylinders to obtain the
horizontal lower length bounds  for \eqref{ehnme},  in the same
way as in the case $n=2$.
  Moreover,  the $n$-dimensional
vertical totally geodesic planes (\remref{rm1} (\ref{geo})) can be
employed,   to obtain
the desired horizontal upper bounds for \eqref{ehnme}, working similarly as in the case  $n=2.$
\end{proof}

\begin{remark}
{\em The proof shows that the estimate hold if we allow  the solutions to be nonnegative on the boundary}.

\end{remark}
%\newpage
\section{ Uniform boundary gradient estimates and\\ modulus of
continuity}\label{sbe}

In this section  we build uniform {\em barriers} at any point of
the  boundary of a bounded convex smooth domain, for a   positive
smooth solution $g$ of the $\ep$-horizontal minimal equation
 $C^ 1$ up to the boundary.  We obtain in fact these a
priori gradient estimates in the sprit of \cite{asi}, on account
of  the techniques in \cite{GT}. We also construct analytic
barriers to get an uniform {\em modulus of continuity}   along the
boundary   of a bounded convex   domain, for  a positive solution
$g$ of the $\ep$-horizontal minimal equation $C^ 0$
up to the boundary.
\begin{definition}\label{ncon}
{\em We say that a $C^0$ domain  $\Omega\subset\partial_\infty
\Hip^n\times\R$ is  convex if, for any $p\in \partial \Om,$ $\Om$
lies in one side of some $n-1$-plane $\Pi$ of $\partial_\infty
\Hip^n\times\R$ passing through $p$, {\em i.e} $p\in \Pi\cap
\partial \Om$ and $\Pi \cap \Om =\emptyset$. }
\end{definition}
We   need  now   the definition of the quantity $R(\Om,f)$
raised in \defref{qua}.

%\newpage
\begin{theorem}[{\bf{\em Uniform boundary gradient estimates I}}]\label{ge1}
${}$
 Let $\Om\subset\partial_\infty \Hip^n\times\R$ be a $C^2$
bounded domain. Let $\vphi\in C^2(\overl{\Om})$ be a positive
function. Let $g\in C^2(\Om)\cap C^1(\overl{\Om})$ be a positive
solution of the $\ep$-horizontal minimal equation \eqref{ehnme}
in $\Om,$ such that $g=\vphi$ on $\Ga=\partial \Om.$ Assume that
$\Om$ is convex.
Then, the following estimate holds.
\begin{multline}\label{bge}
\hfill \max\limits_\Ga |D g|\leqs C\hfill\\
\hfill \text{where}\quad C=C(\min\limits_{\Ga} g,
\max\limits_{\Ga} g,r(\Om),\max \limits_{ \Om}
\vphi, \max \limits_{ \Om} |D\vphi|, \max \limits_{\Om}
|D^2\vphi|).\hfill
\end{multline}
\end{theorem}
\begin{proof}
Let $g$ be a solution of the $\ep$-horizontal minimal equation as
in the statement of the Theorem.

 It suffices to get a priori  estimates for the normal
derivatives at any point $p\in \partial \Om.$ We obtain first the
upper bound  for the normal derivatives constructing an upper
barrier.

We define
\begin{equation}
\psi(d):=\frac{1}{b_1} \ln \left(1 + \frac{ e^{ (R(\Om,f) +
\max\limits_\Om \vphi)\,b_1}-1}{\de_1} d\right )
\end{equation}
 where $b_1$ and $\de_1$ are positive constants to be defined later.

Let $p\in \Om$ and let $\Pi$ be the $n-1$-plane passing through
$p$ as in
\defref{ncon} and let $d(q)=d(q, \Pi), q\in \Om,$ where $d(q, \Pi)$ is the Euclidean distance from $q$ to $\Pi.$ We define $v(q):=
\psi(d(q)), q\in \Om$. Of course, $v(p)=0.$

Let $w=\vphi +v$ and let $\bMc_\ep$ be the strictly elliptic
operator given by
\begin{multline}\label{nelo}
\hfill \bMc_\ep(w):= \sum\limits_{k=1}^{n-1}\left[
g^2\left(1+w_{x_1}^2+\cdots+\widehat{w_{x_k}^2}+\cdots
w_{x_{n-1}}^2\right) +w_t^2+\frac{\ep}{n-1}\right]w_{x_k
x_k}\hfill\\ +
\left(1+\sum\limits_{k=1}^{n-1}w_{x_k}^2\right)w_{tt}
-2\sum\limits_{k=1}^{n-1} w_{x_k} w_t w_{x_k t}
-2g^2\sum\limits_{1\leqs j<k\leqs n-1} w_{x_j} w_{x_k}
w_{x_jx_k}\hfill\\
 +(n-1) g \left( 1 + \sum\limits_{k=1}^{n-1}
w_{x_k}^2\right)
+(n-2)\frac{w_t^2}{g}\hfill\\
:= a_{ij}(g, D w )D_{ij} w +(n-1) g \left( 1 +
\sum\limits_{k=1}^{n-1} w_{x_k}^2\right)
+(n-2)\frac{w_t^2}{g}\\
 \hfill \text{(using summation convention)}\hfill
\end{multline}
\vskip2mm
In view of  inequalities \eqref{ell} and \eqref{ell2}  we have:
\vskip1mm
\begin{multline}\label{nbe1}
\min\{1, g^2\}|\xi|^2\leqs a_{ij}(g, D w )\xi_i\xi_j\leqs\hfill\\
\hfill \hphantom{0000000000000}\Bigl(2 + g^2 (n-1)
+\left(\max\{1,g^2\} (n-2) +1\right)|D w|^2\Bigr)|\xi|^2\,
\text{in}\,\Om \hfill
\end{multline}
\vskip2mm
Now by invoking  the horizontal length estimates (\lemref{hhen}),
it follows that

\vskip1mm
\begin{multline}\label{nbe2}
\min\{1, \min\limits_\Gamma g^2\} |\xi|^2\leqs a_{ij}( g,D w
)\xi_i\xi_j \quad \text{in}\, \Om\hfill\\
\text{Set}\, \La_1:=2 + g^2 (n-1) +\left(\max\{1,g^2\} (n-2) +1\right)|D
w|^2\hfill\\
=2 + g^2 (n-1) +\left(\max\{1,g^2\} (n-2) +1\right)|D \vphi
+ D v|^2\,\text{in}\, \Om.\\
\end{multline}
\vskip2mm
We deduce the following inequalities:
\newpage
\begin{multline}\label{nbe3}
\La_1 <2 + R^2(\Om,f) (n-1) +\left(\max\{1,R^2(\Om,f)\} (n-2)
+1\right)|D \vphi + D v|^2\hfill\\
 \hphantom{\La00}<2 + R^2(\Om,f) (n-1) +2\left(\max\{1,R^2(\Om,f)\} (n-2)
+1\right)(|D \vphi|^2 +| D v|^2)\hfill\\
\hphantom{\La00}\leqs2 + R^2(\Om,f) (n-1)
+2\left(\max\{1,R^2(\Om,f)\} (n-2) +1\right)|D \vphi|^2\hfill\\
\hfill+2\left(\max\{1,R^2(\Om,f)\} (n-2)+1\right) | D v|^2\hfill\\
\hphantom{\La00}\leqs\Bigl(2 + R^2(\Om,f) (n-1)
+2\left(\max\{1,R^2(\Om,f)\} (n-2) +1\right)|D \vphi|^2\Bigr)| D v|^2\hfill\\
\hfill \qquad\qquad+2\left(\max\{1,R^2(\Om,f)\} (n-2)+1\right) | D v|^2,\qquad \text{whenever}\quad | Dv|\geqs 1\hfill\\
\hphantom{\La00}\leqs\Bigl(2 + R^2(\Om,f) (n-1)
+2\left(\max\{1,R^2(\Om,f)\} (n-2) +1\right)(|D \vphi|^2 +1)\Bigr)| D v|^2,\hfill\\
\hphantom{\La00000000000000000000000000000000000}\qquad \text{whenever}\quad | Dv|\geqs 1\hfill\\
\end{multline}
Set
\begin{equation}\label{alp}
 \dd \a_1:=\frac{\Bigl(2 + R^2(\Om,f)
(n-1) +2\left(\max\{1,R^2(\Om,f)\} (n-2) +1\right)(\max\limits_\Om|D
\vphi|^2 +1)\Bigr)}{\min\{1, \min\limits_\Gamma g^2\} }.
\end{equation}
By combining    \eqref{nbe2},       \eqref{nbe3} and \eqref{alp} we have
\begin{multline}\label{2nbe3}
\La_1<\a_1\min\{1, \min\limits_\Gamma g^2\} |Dv|^2\hfill\\
\hphantom{\La00000}< \a_1\, a_{ij}(g, D \vphi + Dv) D_i v D_j v,\qquad \text{whenever}\quad | Dv|\geqs 1\hfill\\
\hphantom{\La00000}<\a_1\,\Fc,\qquad \text{whenever}\quad | Dv|\geqs 1\hfill\\
 \text{where}\quad \Fc:=a_{ij}(g, D \vphi + Dv)D_i v D_j v.\hfill
\end{multline}
Now set  $\La_2:= 1 + |D w|^2= 1 + | D\vphi + D v|^2.$ In the same
way as in proof of the inequalities \eqref{nbe3}, we deduce
~\begin{multline}\label{nbe32}
\La_2<\a_2 \Fc,\qquad \text{whenever}\quad | Dv|\geqs 1\hfill\\
\text{where}\, \a_2=\frac{ 3 + 2 \max\limits_\Om |
D\vphi|^2}{\min\{1, \min\limits_{\Ga}g\}}\hfill
\end{multline}
 From \eqref{nelo}, \eqref{nbe3}, and from the definitions
of $w$, $\Fc$, $\a_1$ and $\a_2$, we get
\begin{multline}\label{n}
\bMc_\ep (\vphi +v)= a_{ij} D_{ij} v + a_{ij} D_{ij}\vphi +(n-1)g
\left( 1 + \sum\limits_{k=1}^{n-1} (\vphi_{x_k}
 +v_{x_k})^2\right)\\
\qquad \qquad\qquad \quad+(n-2)\frac{(\vphi_t +v_t)^2}{g}\hfill\\
\hphantom{nnnn}=a_{ij} \psi''(d)  D_id D_j d +a_{ij} D_{ij}\vphi
+(n-1) g \left( 1 + \sum\limits_{k=1}^{n-1} (\vphi_{x_k}
+v_{x_k})^2\right)\\
\qquad \qquad\qquad \quad +(n-2)\frac{(\vphi_t +v_t)^2}{g}\hfill\\
\qquad\qquad\text{(since $d$ is linear)}\hfill\\
 \hphantom{nnnn}\leqs a_{ij} \frac{\psi''(d)}{\psi'(d)^2}
D_iv D_j v + \La_1 |D^2\vphi| +(n-1) g \left( 1 +
\sum\limits_{k=1}^{n-1}
(\vphi_{x_k} +v_{x_k})^2\right)\\
\qquad \qquad\qquad \quad +(n-2)\frac{(\vphi_t +v_t)^2}{g} \hfill\\
 \hphantom{nnnn}< a_{ij} \frac{\psi''(d)}{\psi'(d)^2}
D_iv D_j v + \La_1 |D^2\vphi|\hfill \\
\qquad \qquad\qquad \quad+\max\left\{(n-1)\max\limits_\Om
g,\,\frac{n-2}{\min\limits_\Ga g}\right\}(1 + |D \vphi + D v|^2)
 \hfill\\
  \hphantom{nnn}  \leqs -b_1 a_{ij}
D_iv D_j v + \La_1 |D^2\vphi|\hfill\\
 \qquad \qquad\qquad \quad +\max\left\{(n-1)\max\limits_\Om
g,\,\frac{n-2}{\min\limits_\Ga g}\right\}(1 + |D \vphi + D v|^2)\hfill\\
 \hphantom{nnnn}< -b_1\Fc
 +  (\max\limits_\Om |D^2\vphi|) \a_1\, \Fc
 + \max\left\{(n-1)\max\limits_\Om
g,\,\frac{n-2}{\min\limits_\Ga g}\right\}\a_2\, \Fc \hfill\\
\end{multline}
\begin{multline*}
\hphantom{nnnn}<\left(-b_1
 +   \a \right)\, \Fc,\quad\text{whenever}\quad | Dv|\geqs 1, \hfill\\
 \hphantom{nnnnnn}\text{where}\, \a:= (\max\limits_\Om |D^2\vphi|) \a_1\,
 + \max\left\{(n-1)\max\limits_\Om
g,\,\frac{n-2}{\min\limits_\Ga g}\right\}\a_2\hfill
\end{multline*}
Now we choose $b_1$ and $\de_1$ such that

\begin{enumerate}

\item $b_1\geqs  \a\,  .$

 \item $\dd \frac{e^{ (R(\Om,f) +\max\limits_\Om \vphi)\, b_1}-1}{ \de_1}\geqs b_1 e^{( R(\Om,f)+\max\limits_\Om \vphi)
 b_1}.$

  With these choices of $b_1$ and $\de_1$ we get that $|D
 v|=\psi'(d) \geqs 1$ if $d \leqs \de_1.$
\end{enumerate}

  Define $\Nc =\{ q\in \Om;\, d(q)< \de_1\}.$ From the choices of $b_1$ and $\de_1$ above we
 deduce that $w=\vphi +v$ is a positive supersolution of \eqref{nelo}; that is, $\bMc_\ep(\vphi +v) <0$ in $\Nc.$   Observe that the
 linear elliptic operator $\Lc$  given by $\Lc[g -(\vphi + v)]:= \bMc_\ep(g)-\bMc_\ep(\vphi +v)= \HMc_\ep(g)-\bMc_\ep(\vphi +v)=-\bMc_\ep(\vphi +v) >0$
 in $\Nc,$ since $g$ is a positive solution of the $\ep$-horizontal minimal equation, by assumption. Furthermore, $\Lc[g -(\vphi + v)]$ satisfies the Hopf
 maximum principle. Recall now that $g<R(\Om,f),$ by \lemref{hhen}. On $\partial \Nc \cap \Om$ we get $g <R(\Om,f)\leqs\vphi +  R(\Om,f)+  \max\limits_\Om
 \vphi= \vphi + v.$  On $\partial \Om \cap\partial \Nc$ we have
  $g=\vphi \leqs \vphi +v.$ Hence  $g \leqs \vphi +v$ in
 $\Nc,$ by the maximum principle. Moreover, $g(p) =\vphi(p)
 +v(p).$ Therefore, $w=\vphi + v$ is an upper barrier (since
 $g(q)\leqs w(q),\, q\in \Nc$ and $g(p)=w(p)$). From this,  the upper bound for the normal derivatives follows.
 We obtain henceforth the desired   boundary gradient upper
 bound.

To obtain the lower bound  for the normal derivatives, we will
construct a lower barrier. Consider $\vphi - v$. Note that
$\vphi(p)-v(p)=\vphi(p)=g(p)
>0.$ Notice that there exists a connected part  $N$ of $\Nc$, containing $p$,  such that $\vphi - v >0$
in $N.$ We derive the following computations.

\begin{multline}\label{be5}
\hfill\bMc_\ep (\vphi-v) = -a_{ij} D_{ij} v +a_{ij} D_{ij} \vphi
+(n-1) g \left( 1 + \sum\limits_{i=1}^{n-1} (\vphi_{x_k}
-v_{x_k})^2\right)\hfill\\
\qquad\qquad\qquad\,+(n-2)\frac{(\vphi_t -v_t)^2}{g}\hfill\\
\hphantom{nnnnnnn}=b_1 a_{ij} D_i vD_j v + a_{ij}D_{ij}\vphi +
(n-1) g \left( 1 + \sum\limits_{i=1}^{n-1} (\vphi_{x_k}
-v_{x_k})^2\right)\hfill\\
\qquad\qquad\qquad\,+(n-2)\frac{(\vphi_t -v_t)^2}{g}\hfill
\end{multline}
\begin{multline*}
\hphantom{nnnnnnn}\geqs b_1\min\{1, \min\limits_\Ga g^2\}
|Dv|^2-\max\limits_\Om |D^2 \vphi|\La_1\hfill\\
\hphantom{nnnnnnnnn}\text{Thus}\hfill\\
 \hfill\bMc_\ep (\vphi-v) >0, \quad \text{whenever}\quad |D v|
\geqs 1,\hfill
\end{multline*}

\vskip1mm
 if $b_1$ is chosen  big enough. So with this choice of
$b_1$, it follows that $\vphi - v$ is a subsolution of
\eqref{nelo}. Noticing that on $\partial \Nc\cap\Om,$
$v=\psi(\de_1)= \max\limits_\Om \vphi +R(\Om,f),$ we deduce that
 on $\partial \Nc\cap\Om$ we
have  $ \vphi -v= \vphi -\max\limits_\Om \vphi -R(\Om,f)<0<g.$
 And on $
\partial \Om \cap \partial \Nc,$ we have $\vphi-v \leqs \vphi=g,$
hence $\vphi -v \leqs g$ on $\Nc$, by the maximum principle  and
$g(p) =\vphi(p) -v(p).$ We thus infer the a priori  lower bound
for the normal derivatives.  Therefore  we obtain the desired a
priori boundary gradient estimates. This completes the proof of
the Theorem.
\end{proof}

\begin{remark}\label{slo}
{\em For  a $C^ {2, \a}$  domain $\Om$ for some $0<\a<1$ whose
boundary has positive mean curvature, we can find a  simpler and
more geometric a priori lower bounds  for the gradient. In fact,
these domains satisfies a {\em boundary slope condition} and the
Euclidean $n$-planes  can be used as lower barriers, since they
are subsolutions of the  $\ep$-horizontal mean equation
\eqref{ehnme}.}
\end{remark}

Recall now that the minimal equation in Euclidean space   is given
by

 \begin{equation}\label{min}
 \diver\left(\frac{\nabla u}{W(u)}\right):= \sum\limits_{i=1}^n\frac {\partial}{\partial
x_i}\left(\frac{ u_{x_i}}{\sqrt{1 +\|\nabla u\|^2_{{\R^n}}}}
\right)=0
 \end{equation}

\begin{definition}\label{euc}

Let $\Om \subset \R^n$ be a $C^ {2, \a}$ bounded convex domain for
some $0<\a<1,$ with boundary $\Ga.$ Let $f\in C^ {2, \a}(\Ga)$. We
define by $\wdh{\vphi}\in  C^ {2, \a}(\overl{\Om})$ the unique
minimal solution of the minimal equation in Euclidean space in $
\Om$, taking the prescribed boundary data $f$ on $\Ga$ \cite{JS2}.
\end{definition}

\vskip2mm

We recall that the  maximum principle and the use of the Euclidean
$n$-planes as barriers ensures that $ \min\limits_\Ga f
<\wdh{\vphi}< \max\limits_\Ga f$ in $\Om.$

 The following  Lemma is well-known and we  will use it in the proof of the next
 Lemma. We will write a proof for completeness.

\begin{lemma}\label{emine}
  Let $\Om\subset\R^n$ be a $C^ {2, \a}$ bounded convex
domain for some $0<\a<1$. Let $\Ga=\partial \Om$ and let $f\in
C^{2, \a}(\Ga)$ be a positive function.

For each $s\in [1/2,1]$, let $\wdh{\vphi_s}\in C^ {2,
\a}(\overl{\Om})$ be the unique solution of \eqref{min} taking the
prescribed boundary value data $(2s -1) f + 2(1-s)\min\limits_\Ga
f $ on $\Ga.$ Then there exists a constant $K$, independent of
$s$, such that

\begin{equation}\label{eto}
\max\limits_{\overl{\Om}}|D \wdh{\vphi_s}| +
\max\limits_{\overl{\Om}}|D^2\wdh{\vphi_s}|\leqs K
\end{equation}

\end{lemma}

\begin{proof} For the readers convenience, we   outline a
proof, as follows. First, note that the Extension Lemma
\cite[Lemma 6.38]{GT} provides a $C^{2,\, \a}(\overl{\Om})$
extension $\vphi_s$ of $(2s -1)f  + 2(1-s)\min\limits_\Ga f\,$
such that $\max\limits_{\overl{\Om}}|{\vphi_s}|+
\max\limits_{\overl{\Om}}|D {\vphi_s}| +
\max\limits_{\overl{\Om}}|D^2{\vphi_s}|  $ is bounded by a
constant $C$ independent of $s,\, s\in [1/2, 1].$

   As we have
observed before, the maximum principle yields\\
$\min\limits_\Ga f\leqs\wdh{\vphi_s}\leqs \max\limits_\Ga f,$
since this inequality occurs on the boundary. Note also that  that
the first eigenvalue of the  matrix associated  to equation
\eqref{min} is $1$.
 Now using the extension $\vphi_s$, we can follow, step by step, the
proof of \thmref{ge1}, to ensure the bounds for the first
derivatives of $\wdh{\vphi_s}$, independent of $s.$ Then by
applying the global a priori H\"older estimates of Ladyzhenskaya
and Ural'tseva \cite{L-U},  \cite[Theorem 13.7]{GT}, we have global Hölder a
priori estimates for the first derivatives of $\wdh{\vphi_s}$,
independent of $s.$ Finally, the global a priori Schauder
estimates \cite[Theorem 6.6]{GT} shows that the $C^{2,\,
\a}(\overl{\Om})$ norm $|\wdh{\vphi_s}|_{C^{2,\,
\a}(\overl{\Om})}$, is bounded by a constant $K$, independent of
$s,\, s\in [1/2, 1]$. This gives the desired estimate.

\end{proof}

Let $\ell: [0, 1]\goto \R$ be a continuous function satisfying
$0\leqs\ell(s)\leqs 1,$   $\ell(1)=1$  and $\ell(s)=0,\, s\in [0,
1/2].$ For each $s\in [0, 1],$ let us now turn attention to the
positive solutions of the following P.D.E:

\begin{multline}\label{fme}
 \sum\limits_{k=1}^{n-1}\left[
g^2\left(1+g_{x_1}^2+\cdots+\widehat{g_{x_k}^2}+\cdots
g_{x_{n-1}}^2\right) +g_t^2+\frac{\ep}{n-1}\right]g_{x_k
x_k}\hfill\\ +
\left(1+\sum\limits_{k=1}^{n-1}g_{x_k}^2\right)g_{tt}
-2\sum\limits_{k=1}^{n-1} g_{x_k} g_t g_{x_k t}
-2g^2\sum\limits_{1\leqs j<k\leqs n-1} g_{x_j} g_{x_k}
g_{x_jx_k}\hfill\\
 +\Bigl[(n-1) g \left( 1 + \sum\limits_{k=1}^{n-1}
g_{x_k}^2\right)
+(n-2)\frac{g_t^2}{g}\Bigr]\ell (s)\hfill\\
\hphantom{0000000000}=0\hfill
\end{multline}

Note that for  $s=1,$  equation \eqref{fme} reduces to the
$\ep$-horizontal minimal equation.

For later purposes, we need a  slight refined generalization of
the boundary gradient estimates for $C^ {2, \a}$  domains and
boundary data, on account of  \lemref{emine} and the above
observations.

\begin{theorem}[{\bf{\em Uniform boundary gradient estimates II}}]\label{ge2}

  Let $\Om\subset\partial_\infty \Hip^n\times\R$ be a
$C^ {2, \a}$ bounded convex domain for some \\$0<\a<~1$. Let
$\Ga=\partial \Om$ and let $f\in C^ {2, \a}(\Ga)$ be a positive
function. For each $s\in [1/2,1]$, let $g_s\in C^{2,
\a}(\overl{\Om})$ be a  positive solution of \eqref{fme} in $\Om,$
such that $g_s=  (2s -1)f + 2(1-s)\min\limits_\Ga f$ on $\Ga.$

Then, for each $s\in [1/2, 1]$ the following estimate holds.

\begin{multline}\label{bge*}
\hfill \min\limits_\Ga f\leqs \max\limits_{\overl{\Om}} g_s< R(\Om,f)\quad \text{on} \quad\overl{\Om}\hfill\\
 \text{where} \quad R(\Om,f) \quad \text{is given by \defref{qua}}\hfill\\
\text{and}\hfill\quad \max\limits_{\Ga}|D g_s|\leqs C \quad \hfill\\
\hfill \text{where}\quad C=C(\min\limits_{\Ga} f, \max
\limits_{\Ga} f,r(\Om),K),\hfill
\end{multline}
 where $K$ is given by \lemref{emine}.

\end{theorem}

\begin{proof}   Let $\wdh{\vphi_s}$ be the unique minimal Euclidean extension
 of $ (2s -1)f + 2(1-s)\min\limits_\Ga f$
satisfying equation \eqref{min}.

Notice that, as in the proof of the \thmref{ge1}, we have the
following ingredients:

\noi {\em   1}:  First, it is routine to check  that $
\min\limits_\Ga f\leqs g_s|_\Ga \leqs \max\limits_\Ga f.$ On the
other hand, observe that  the Euclidean $n$-planes are
subsolutions of \eqref{fme}. Moreover, the vertical geodesic
$n$-planes (\remref{rm1} (2)) are  still supersolutions of
\eqref{fme}. Consequently, the length estimate inferred in
\secref{she} hold with  exactly  the same statement as in
\lemref{hhen}. Thus, by invoking the length estimate, it follows
that $\min\limits_\Ga f< g_s< R(\Om,f)$ on $\Om$ and  $\min\limits_\Ga f\leqs g_s< R(\Om,f)$ on $\overl{\Om}.$

\noi {\em   2}: Secondly, taking into account  $0\leqs
\ell(s)\leqs 1,$ inequality \eqref{eto}, and inequalities
\eqref{ell} and \eqref{ell2}, we are able to follow the procedure
of the proof of \thmref{ge1}  to obtain the desired estimate.

Noticing that  in view of \lemref{emine} and its proof, we have
$K$ in the place of $\max\limits_{\Om}|D \wdh{\vphi_s}| +
\max\limits_{\Om}|D^2\wdh{\vphi_s}|$ and  $\max\limits_\Ga f$
instead of $\max\limits_\Om \wdh{\vphi_s}$ in the estimate.

This accomplishes the proof of the theorem.

\end{proof}

Next, we show that a positive solution $g\in C^2(\Om)\cap
C^0(\overl{\Om})$ of \eqref{ehnme} has a uniform modulus of continuity
along a bounded convex domain $\Om.$ \vskip2mm
 Let $\Ga\subset\partial_\infty \Hip^n\times\R$ be a bounded convex
curve  and let $f\in C^ {0 }(\Ga).$ Given $\e>0$, the continuity
of $f$  yields the existence of a positive constant $\de_0>0$,
 such that $|f(p')-f(p)|<\e/3, $ if $|p' -p| <\de_0,$ $p' ,p\in
 \Ga.$

 We now define  the barriers $\vphi^ \pm$ at $p\in \Ga$  that we need in the next
 theorem. Let

 \begin {equation}\label{bar}
\vphi^ \pm(q):= f(p) \pm\frac{\e}{3} \pm  R(\Om,f)\frac{ \ln ( 1 + |
q -p|^ 2)}{\ln (1 + \de_0^ 2)}, \quad q\in \Om
 \end{equation}
 where $\de_0=\de_0
(f)$ is defined in the above paragraph.

\begin{theorem} [{\bf{\em Uniform modulus of continuity}}]\label{mod}
  Let $\Om\subset\partial_\infty \Hip^n\times\R$ be a $C^ {0}$
bounded domain. Let $\Ga=\partial \Om$  and let $f\in C^ {0
}(\Ga)$ be a positive function. Let  $g\in C^{2}(\Om) \cap
C^{0}(\overl{\Om})$ be a positive solution of the $\ep$-
horizontal minimal equation \eqref{ehnme} in $\Om,$ such that
$g=f$ on $\Ga.$ Assume that $\Om$ is convex.

Then, it follows that given $\e>0$ there is $\de >0$ such that if
$ p\in \partial \Om$ and  $q\in \Om, $ satisfy $|q-p| <\de,$ then
$|g(q)-g(p)| <\e,$ where \\$\de =\de(\e,\de_0, \min\limits_{\Ga}
g, \max \limits_{\Ga} g, r(\Om), \max \limits_{ \Om}
|D\vphi^ \pm|, \max \limits_{\Om} |D^2\vphi^ \pm|).$

\end{theorem}

\begin{proof}
The proof is carry out by mimicking the proof of the
 boundary gradient estimates, using  the barriers $\vphi^ \pm$  (instead of the function $\vphi$). Let $\e >0$.

Notice that $f\leqs \vphi^ +$ on $\Ga$ and  $f\geqs \vphi^ -$ on
$\Ga.$

We need to prove that we  can  find a suitable constant $\de
>0$ such that if $|q-p| <\de,\ q\in \Om, p\in \Ga,$ then both
inequalities  $ g(q) < f(p) + \e $ and $ g(q)> f(p) -\e$ hold.

To obtain the first inequality we use the barrier $w^+=\vphi^ + +
v^+$, where the function $v^+$  is  defined by $v^+(q):=
\psi^+(d(q)), q\in \Om$, where
\begin{equation}
\psi^+(d):=\frac{1}{b_1} \ln \left(1 + \frac{ e^{ R(\Om,f)
\,b_1}-1}{\de_1} d\right )
\end{equation}

Indeed working exactly as in the proof of \thmref{ge1}, we can
deduce that $w^+$ is a supersolution of \eqref{nelo} in $\Nc,$ and
$g\leqs \vphi^+ + v^+$ in $\Nc,$ where $\Nc$ is defined in the
proof of \thmref{ge1}. Next,  we are able to choose $\de>0$ small
enough, $\de<\de_0$, $\de< \de_1,$ such that if $|q-p|<\de,$ then
$0\leqs\vphi^ +(q) < f(p)+ 2\e/3$ and $0<v^+(q)<\e/3.$ We infer
therefore that if $|q-p|<\de$ then $ g(q) < f(p) + \e,$ as
desired.

In order to achieve the proof of the theorem, we now define $v^-(q):=\\
\psi^-(d(q)), q\in \Om$, where
\begin{equation}
\psi^-(d):=\frac{1}{b_1} \ln \left(1 + \frac{ e^{ (R(\Om,f)
+\max\limits_{\Om}|\vphi^-|)  \,b_1}-1}{\de_1} d\right )
\end{equation}

Then, we can use the barrier $w^-=\vphi^- - v^-$, working as
before, to deduce that $w^-$ is a subsolution of \eqref{nelo} in
$\Nc,$ and finally, to obtain that $g(q)> f(p)-\e$. This
accomplishes the proof of the Theorem.

\end{proof}

\section{Uniform  global gradient estimates}\label{sge}

In this section we are able to obtain uniform a priori global
gradient estimates for the $\ep$-horizontal minimal equation in
two independent variables.
 The a priori interior gradient bound for the minimal equation in Euclidean space in two independent variables was established   by R. Finn \cite{Fin}. The $n$-dimensional case was done by Bombieri, De Giorgi and Miranda \cite{B-G-M}.  Equations of minimal type were first studied  by Finn \cite{Fin}. He established   a priori estimates for the gradient of a solution. Then H. Jenkins and J. Serrin \cite{JS2} obtained curvature estimates for some kind of such equations.

  In his fundamental paper L. Simon   \cite{Si}, derived  a priori gradient estimates, for general mean curvature type equations in two variables. His results has been be applied to many mean curvature equations in two variables in several  spaces \cite{asi}, \cite{E-N-SE}. However, due to the geometry of the ambient space,  the horizontal minimal equation does not fit the structure conditions of the equations of mean curvature type, established by L. Simon. It does not match either the structure conditions in \cite[Theorem 5.2]{GT}.

 Moreover,  the technique developed by L. Cafarelli, L. Nirenberg and J. Spruck in \cite{CNS} to obtain a priori global gradient estimates (depending on the a priori $C^0$ estimates and on the a priori $C^1$  boundary gradient estimates),    cannot be applied  to the present situation. On the contrary, the method in \cite{CNS} has been applied to the mean curvature equation in hyperbolic space in  \cite{N-SE} and \cite{Ba-SE} and to a vertical constant mean curvature equation in warped product \cite{D-R}.  We notice that, in some cases, the a priori global $C^1$ estimates depends on   the derivation of both a priori $C^0$ estimates and  a priori $C^1$ boundary gradient estimates, see, for instance, \cite{N-Sp} and \cite{G-SE}.

   Finally, we remark that the general a priori gradient estimates inferred by J. Spruck  \cite{Sp}, are adapted for the vertical mean curvature equation to study many problems in several product spaces $M \times \R,$ where $M$ is a Riemannian manifold  \cite{Be-SE2}, \cite{G-R}, \cite{SE-T4}, \cite{C-C}.

\vskip2mm

 Next, we follow the techniques derived in
\cite{Ba}. We do the analysis in two independent variables, for the reasons that will be clear in the proof.

Notice that certain horizontal  minimal graphs in two independent variables  are {\em Killing graphs} \cite{Sa2}.  That is they are graphs with respect to   the coordinates system given by the $1$-parameter group of isometries $(x, y, t)\goto (\la\, x, \la\, y, t),\, \la >0$ of  $\hip^2\times \R$. This group is constituted of hyperbolic translations of the ambient space. Noticing also that by
employing the translations $(\overl{x}, \overl{y}, t)\goto (\la
\overl{x}, \la \overl{y}, t):=(x, y, t),$ we obtain another
function $y=g_\la(x, t),\, (x, t),$ satisfying the equation \eqref{eh2me}  for
$\ep\la^ 2$ in the place of $\ep$ (see \eqref{ehnmel}).

 Due to the techniques employed, we are forced to make a  strong  constraint on the horizontal length to obtain the uniform global $C^1$ gradient estimates. Nevertheless, the assumption is invariant by hyperbolic translations, so it is
  compatible with the geometry of the ambient space.

\vskip2mm

We recall that if $n=2$ the $\ep$-horizontal minimal equation becomes
\begin{equation}\label{eh2me}
\hfill \HMc_\ep (g)=g_{xx} (g^2 +g_t^2 +\ep) + g_{tt} (1 +g_x^2)
-2g_xg_tg_{xt} +g(1 +g_x^2)=0,\hfill
\end{equation}

\vskip2mm
\begin{theorem}[{\bf{\em Uniform global gradient estimates I}}]\label{gge}
Let $\Om\subset\partial_\infty \Hip^2\times\R$ be a $C^1$  bounded
domain.  Let $g\in C^{2, \a}(\Om)\cap C^1(\overl{\Om})$ be a
positive solution of the $\ep$-horizontal minimal equation
\eqref{eh2me} in $\Om.$ Assume there exist constants $c_1$ and
$c_2$, and $c_3,$ independent of $\ep,$ such that $0< c_1\leqs
g\leqs c_2$ on $\overl{\Om}$ and $ |D g|\leqs c_3$ on $\partial
\Om.$  Assume further that $\dd c_2 <c_1
+c_1\sqrt{\frac{\pi}{2}}.$ Then there exists a constant $C$
depending only on $c_1, c_ 2$ and $c_3$ such that $ |D g| \leqs C$
on the whole $\overl{\Om}.$

\end{theorem}

\vskip2mm

\begin{proof}

 Let $g$ be a solution of the $\ep$-horizontal minimal equation on $\Om$, $C^ 1$ up to the boundary, such that $0<c_1\leqs g\leqs c_2$ on
 $\overl{\Om}.$

 Assume first that $c_1\leqs 1$.

 By assumption, we have $\dd c_2-c_1 <c_1\sqrt{\frac{\pi}{2}}.$
 We first recall the  elementary identity $\dd \int\limits_{0}^ {+\infty} e^ {-\ga s^
2} \rmd s=\left(\frac{\pi}{2\cdot 2\ga}\right)^{1/2}.$ We now
consider the function $\phi(u)$ given by
\begin{equation}\label{au1}
\hfill \phi(u)= c_1 + \int\limits_{0}^ u e^ {-\ga  s^ 2} \rmd
s\hfill
\end{equation}
where $\ga$ is chosen   such that $ c_1> \frac{1}{\sqrt{2\ga}}>
\sqrt{ \frac{2}{\pi}}\cdot(c_2 -c_1).$ It follows that $c_2 < c_1
+ \left(\frac{\pi}{2\cdot2\ga}\right)^{1/2}.$

\vskip2mm
 Let $c_{12}=c_{12}(c_1, c_2)$ be a  constant such that $\dd c_1 +
\int\limits_{0}^ {c_{12}} e^ {-\ga s^ 2} \rmd s > c_2.$

With this choice, we are able to write
\begin{multline}\label{aux}
\hfill g(p)=\phi(u(p)),\, u\geqs 0, \, p\in \overl{\Om}\hfill
\end{multline}
 for some auxiliary function $u$ satisfying $u(p)<c_{12}$ on $\overl{\Om}.$
Then, we have the inequalities
\begin{multline}\label{sta}
\hfill e^{ -\ga c_{12}^2} <\phi'(u) \leqs  1\quad\text{ on}\quad \overl{\Om}\hfill\\
\hphantom{00000}\quad\text{ and}\hfill\\
\hfill 0< -\phi''(u)\leqs +2\ga c_{12}\hfill\\
\hfill\frac{-\phi'''(u)\phi'(u) +\phi''^ 2(u)}{\phi'^ 2(u)}=2\ga\hfill\\
\quad\text{ on}\quad \Om.\hfill\\
\end{multline}

\noi

 Next differentiate \eqref{aux} with respect to $x$ and $t$ to
 obtain

\begin{multline}\label{au2}
 \hfill g_x = \phi'(u) u_x,\, g_t =\phi'(u) u_t\hfill\\
 \hfill  g_x^2 +g_t^2 =\phi'^ 2(u) (u_x^2
+u_t^2)\hfill
\end{multline}

Of course, with the aid of equation \eqref{au2}, we infer the
following inequalities:

\begin{multline}\label{egu}
\sqrt{u_x^2 +u_t^2}\leqs \frac{\max\limits_{\partial
\Om} \sqrt{g_x^2 +g_t^2}}{\min\limits_{\overl{\Om}}\phi'(u)},\quad \text{on}\quad \partial \Om \hfill\\
\text{and}\hfill\\
\qquad\sqrt{g_x^2 +g_t^2}\leqs \sqrt{u_x^2 +u_t^2},\quad \text{on}\quad  \overl{\Om}. \hfill\\
\end{multline}

 Thus, if the maximum of $D u$ on $\overl{\Om}$  is
achieved at the boundary then $\max\limits_{\overl{\Om}} | Du|
\leqs  \frac{\max\limits_{\partial\Om} |
Dg|}{\min\limits_{\overl{\Om}} \phi'(u)}\leqs c_3 e^{\ga c_{12}
+\ga c_{12}^2}$ and we are done. Otherwise, we assume that the
maximum of $D u$ is attained at an interior point $p$ of $ \Om$.
Set $ w(q):= \frac{u_x^2(q) +u_t^2(q)}{2}, \, q\in \Om.$ Thus,  it
suffices to infer the desired a priori estimates for the quantity
$w(p)= \frac{u_x^2(p) +u_t^2(p)}{2}.$ Of course, we have

\begin{multline}\label{au0}
\hfill w_x(p)=w_t(p)=0\hfill\\
\hfill\quad\text{and}\quad u_x^2u_{xx} =u_t^2u_{tt} \quad
\text{at}\quad p. \hfill\\
 \end{multline}

We remark that for the deduction of the last equation we used that
we are working only with two independent variables $x, t.$ This
will allow us to write the second derivatives appearing in
\eqref{eh2me}, in term of the first derivatives (at $p$). This is
a crucial step to achieve  the desired global estimates making use
of the present method.

\vskip2mm
 Hereafter, we write $\phi, \phi' ,\phi''.$  Next plug the expressions \eqref{au2} and its derivatives with respect to $x$ and $t$,
 respectively,  in \eqref{eh2me}.  Then  rewrite the $\ep$-horizontal
minimal equation in the form:

\begin{multline}\label{au3}
 \hfill  (\phi^2 +2\phi'^2 w +\ep) u_{xx} +( 1 + 2\phi'^2 w) u_{tt} -\phi'^2(u_x w_x +u_t w_t)      \hfill\\
\hfill +\frac{\phi^2 \phi''}{\phi'}2w -\frac{\phi'' (\phi^2
-1)}{\phi'} u_t^2 +\frac{\phi}{\phi'} (1 + 2 w\phi'^ 2) -
\frac{\phi\phi'^2}{\phi'}u_t^2 +\frac{\phi''}{\phi'}\, \ep\,
u_x^2=0
\end{multline}

Now note that by invoking elliptic regularity $g\in C^ 3(\Om)$. In
fact, $g$ is analytic in $\Om$  by Morrey's regularity theorem
\cite{Mo2}. Set $a_{11}:= \phi^2 + \phi'^2 u_t^2 +\ep, a_{22}= 1
+\phi'^2u_x^2$ and $ a_{12}=a_{21} =-\phi'^2u_x u_t.$ The
following identity will be useful and its verification is easy to
check.

\begin{multline}\label{au4}
(\phi^2 + 2\phi'^2 w+\ep) (u_x u_{xxx} +u_t u_{txx}) + (1 +
2\phi'^2 w)
(u_x u_{xtt} + u_t u_{ttt})\hfill\\
\hfill\hphantom{7777777777}-\phi'^2(u_x^2 w_{xx} + u_t^2 w_{tt} +
2 u_x u_t w_{xt})\hfill \\
\hfill = a_{11} w_{xx} + a_{22} w_{tt}+2a_{12} w_{xt}-(\phi^2 +\ep
+ 2\phi'^2 w)(u_{xx}^2 + u_{xt}^2)-(1 + 2\phi'^ 2 w) (u_{tt}^2 +
u_{xt}^2)
\end{multline}

The conditions  $a_{ij}$ is positive and $D^2 w (p)$ is negative
read

\begin{multline}\label{au4'}
(\phi^2 + 2\phi'^2 w +\ep) (u_x u_{xxx} +u_t u_{txx}) + (1 +
2\phi'^2 w)
(u_x u_{xtt} + u_t u_{ttt})\hfill\\
\hfill\hphantom{7777777777}-\phi'^2(u_x^2 w_{xx} + u_t^2 w_{tt} +
2 u_x u_t w_{xt})\hfill \\
\hfill \quad \leqs 0 \quad \text{at} \quad p.\hfill
\end{multline}

 Now  differentiate the P.D.E \eqref{au3} with respect to $x$
(respectively  differentiate with respect to $t$) and  multiply
the resulting equation by $u_x$ (respectively  multiply by $u_t$).
Adding the two equations thus obtained and using \eqref{au0} and
 \eqref{au4'}, we infer

\begin{multline}\label{au5}
 8\phi'\phi'' w^2(u_{xx} + u_{tt}) + 4 \phi\phi'w\, u_{xx}-2\phi'^2u_t^2 w + 6\phi'' \phi u_x^2 w +2w\left(\frac{\phi'^2-\phi''\phi}{\phi'^2}\right)\hfill\\
\hfill +4 w^2\left[
-\phi^2\left(\frac{-\phi'''\phi'+\phi''^2}{\phi'^2}\right)
+\phi'^2\right] +2u_t^ 2 w \left[ (\phi^2
-1)\left(\frac{-\phi'''\phi'+\phi''^2}{\phi'^2}\right)
\right] \hfill\\
\qquad -\left(\frac{-\phi'''\phi'+\phi''^2}{\phi'^2}\right)\,\ep\,
2 w u_x^2 \geqs 0 \quad \text{at} \quad p.\hfill
\end{multline}

Hence,  according to the foregoing we find

\begin{multline}\label{au6}
 (8\phi'\phi'' w^2)(u_{xx} + u_{tt}) + (4\phi\phi'w)\, u_{xx}\hfill\\
\hfill +4 w^2\left[ -\min\{\phi^2(u),
1\}\left(\frac{-\phi'''\phi'+\phi''^2}{\phi'^2}\right)
+\phi'^2\right]\hfill\\
\hfill\hphantom{000} + (6\phi'' \phi u_x^2 w) +
2w\left(\frac{\phi'^2-\phi''\phi}{\phi'^2}\right)\geqs 0 \quad \text{at} \quad p.\hfill \\
\end{multline}
since $\dd \frac{-\phi'''\phi'+\phi''^2}{\phi'^2}=2\ga>0,$
noticing that the last term of inequality \eqref{au5} is negative
.

We must now estimate the quantities $u_{xx}(p)$ and $(u_{xx} +
u_{tt})(p)$ in terms of $\phi, \phi', \phi''$ and $w$ at $p.$
Owing to \eqref{au0}  performing some computations we find that
\eqref{au3} becomes

\begin{multline}\label{au7}
u_{xx}\, \phi'\left[ (\phi^2+\ep +2\phi'^2 w)u_t^2 +(1
+2\phi'^2 w)u_x^ 2\right]\hfill\\
=-u_t^ 2 \phi''\phi^2 u_x^2 -\phi'' u_t^4 -\phi\, u_t^ 2-\phi'^ 2
\phi\, u_t^2 u_x^2 -\phi''\,\ep\, u_x^2u_t^2 \quad \text{at} \quad
p.\hfill
\end{multline}

\noi and
\begin{multline}\label{au8}
(u_{xx} + u_{tt})\, \phi'\left[ (\phi^2 +\ep +2\phi'^2 w)u_t^ 2
+(1
+2\phi'^2 w)u_x^ 2\right]\hfill\\
 \hfill=- 2\phi'' \phi^2 w u_x^2 -\phi'' u_t^2 2w -2 \phi\, w
 -2\phi\phi'^2 wu_x^ 2-2\phi''\ep \,w u_x^2
  \quad \text{at} \quad p.\hfill
\end{multline}

We now perform some calculations to discover

\begin{multline}\label{au9}
\hfill(8\phi'\phi'' w^2)(p)(u_{xx} + u_{tt})(p)\leqs
-\frac{4\phi\phi''w}{\phi'^ 2}(p) -(4\phi\phi''u_x^ 2w)(p)\hfill\\
\quad \text{and}\hfill\\
(4\phi\phi'w)(p)\, u_{xx}(p)\leqs -\frac{4\phi^ 3\phi''w}{\phi'^ 2}(p)-\frac{8\phi\phi'' w}{\phi'^ 2}(p)\\
 \quad \text{since}\quad  \phi' >0 \quad \text{and}\quad \phi''< 0.\hfill
\end{multline}

Inserting \eqref{au9} into \eqref{au6} we then  obtain

\begin{multline}\label{au10}
\hfill 4 w^2(p)\left[ -\min\{\phi^2(u(p)),
1\}\left(\frac{-\phi'''\phi'+\phi''^2}{\phi'^2}\right)(p)
+\phi'^2(p)\right]\hfill\\
\hfill\hphantom{000} +
2w(p)\left[\left(\frac{\phi'^2-\phi''\phi}{\phi'^2}\right)(p)-\frac{6\phi\phi''
}{\phi'^ 2}(p)-\frac{2\phi^ 3\phi''}{\phi'^ 2}(p)\right]
\geqs 0\hfill \\
\end{multline}

 Employing \eqref{sta} into \eqref{au10} we at last infer

\begin{multline}\label{ult}
 4w^ 2(p)\left[ -2\ga\min \{c_1^2,1\} +1\right] \hfill\\
\hfill \hfill\hphantom{00000000} +2 w(p)\left[ 1 +\Bigl( (\ga
+2\ga c_{12})c_2 +(2 \ga +4\ga c_{12})(3c_2 +c_2^ 3)\Bigr )e^
{2\ga c_{12} +2\ga c_{12}^ 2}\right]\geqs 0\hfill
\end{multline}

since $c_1>0.$

But then,

\begin{multline}\label{ult2}
 w^ 2(p)\left[ -2\ga c_1^2 +1\right] \hfill\\
\hfill \hfill\hphantom{00000000} +2 w(p)\left[ 1 +\Bigl( (\ga
+2\ga c_{12})c_2 +(2 \ga +4\ga c_{12})(3c_2 +c_2^ 3)\Bigr )e^
{2\ga c_{12} +2\ga c_{12}^ 2}\right]\geqs 0\hfill
\end{multline}

since we assume that $c_1\leqs 1.$

Finally, we recall that we have chosen $\ga$,  so that $2\ga
>\frac{1}{c_1^ 2}$.  Henceforth, in the light of \eqref{ult2}  we derive the a
priori bounds for $ |D u(p)|$, if $c_1\leqs 1.$

Now assume that $c_1>1$. Choose $\la<1$ such that $\la c_1<1.$
Write  by $\overl{y}=g(\overl{x}, t)$, the solution of
\eqref{eh2me} in $\Om$, taking coordinates $(\overl{x}, \overl{y},
t)\in \hip^2\times\R.$

  By
employing the translations $(\overl{x}, \overl{y}, t)\goto (\la
\overl{x}, \la \overl{y}, t):=(x, y, t),$ we obtain another
function $y=g_\la(x, t),\, (x, t)\in T_\la(\Om)$  (\remref{rm1}
(3)), satisfying the equation \eqref{eh2me} in $ T_\la(\Om) $ for
$\ep\la^ 2$ in the place of $\ep$. In fact, noticing that $
g_\la(x, t)=\la g(\overl{x}, t)=\la g(\frac{x}{\la}, t).$ It
suffices to compute the relations between the first and second
derivatives of $g(\overl{x},t)$ and $g_\la(x, t).$

 We conclude therefore that
$g_\la$ satisfies

\begin{equation}\label{ehnmel} \hfill  (g_\la)_{xx} ((g_\la)^2 +
(g_\la)_t^2 + \ep\la^ 2) + (g_\la)_{tt} (1 +(g_\la)_x^2)
-2(g_\la)_x(g_\la)_t(g_\la)_{xt} +(g_\la)(1 +(g_\la)_x^2)=0 \hfill
\end{equation}

Now the same relations between the first derivatives of
$g(\overl{x},t)$ and $g_\la(x, t),$ ensures   that $ \la^2(
g_{\overl{x}}^2 (\frac{x}{\la}, t) + g_t^2(\frac{x}{\la}, t))\leqs
(g_\la)_x^2(x, t) +(g_\la)_t^2(x, t)$ and $(g_\la)_x^2(x, t)
+(g_\la)_t^2(x, t)\leqs g_{\overl{x}}^2 (\frac{x}{\la}, t) +
g_t^2(\frac{x}{\la}, t).$

Thus

$\max\limits_{\partial T_\la(\Om)} |D g_\la|\leqs
\max\limits_{\partial\Om} |D g|$ and $\max\limits_{\overl{\Om}} |D
g| \leqs\frac{\max\limits_{\overl{ T_\la(\Om)}} |D g_\la|}{\la}.$

In view of $\la c_1\leqs g_\la \leqs \la c_2, $ in $T_\la (\Om),$
we can commence  again (since $\la c_1<1$), to mimic each step of
the  the above procedure to obtain the  desired a priori bounds,
if $c_1> 1.$

  Henceforth, the
a priori global gradient estimates for $|D g|$ is achieved, as
desired. This completes the proof of the theorem.

\end{proof}

In the case of a convex domain, we state the following global
$C^1$ estimates.  First, if $n=2$, equation \eqref{fme} becomes
(recall that $0\leqs \ell(s)\leqs 1, s\in [0,1], \, \ell(1)=1,
\ell(s)=0,\, s\in [0,1/2]$).

\begin{equation}\label{2fme}
g_{xx} (g^2 +g_t^2 +\ep) + g_{tt} (1 +g_x^2) -2g_xg_tg_{xt}
+\ell(s)g(1 +g_x^2)=0,\hfill
\end{equation}

\begin{theorem}[{\bf{\em Uniform global gradient estimates  II}}]\label{ggc}
${}$

Let $\Om\subset\partial_\infty \Hip^2\times\R$ be a $C^ {2, \a}$
bounded convex domain   for some \\$0<\a<1$. Let $\Ga=\partial
\Om$ and let $f\in C^ {2, \a}(\Ga)$ be a positive function. Let
$\ep\in [0,1]$. Given $s\in [1/2, 1]$, let $g_s\in C^{2,
\a}(\overl{\Om})$ be a  positive solution of equation \eqref{2fme}
in $\Om,$ such that $g_s=  (2s -1)f + 2(1-s)\min\limits_\Ga f$ on
$\Ga.$ Assume further that $\dd R(\Om,f) \leqs\min\limits_\Ga f
+\min\limits_\Ga f\sqrt{\frac{\pi}{2}}$ {\em (where $R(\Om,f)$ is
given by \lemref{hhen})}.

Then, the following estimate holds.

\begin{multline}\label{gge*}
\hfill \max\limits_{\overl{\Om}} |D g_s|\leqs C\hfill\\
\hfill \text{where}\quad C=C(\min\limits_{\Ga} f, \max
\limits_{\Ga} f,\max\limits_{ \Om}x-\min\limits_{\Om} x, K).\hfill
\end{multline}

 where $K$ is given by \lemref{emine}.

\end{theorem}

\begin{proof}
We summarize the proof as follows.
 In view of \thmref{ge2}, it is
routine to check that the proof of the global gradient
(\thmref{gge}) is   valid for equation \eqref{2fme}  taking
$c_2<R(\Om,f)$ and $c_1=\min\limits_\Ga f,$  noticing that $0\leqs
\ell(s)\leqs 1$. Indeed, we obtain the same  a priori bounds  as
given by equations  \eqref{ult} and \eqref{ult2}.

\end{proof}

\section{ An existence  result}

\vskip4mm

 Our  main goal now is to prove the existence of a solution of a Dirichlet problem for
the  horizontal minimal equation. However, the non strictly ellipticity of the equation imposes a new insight.

 Because of that, we need to consider the family of $\ep$-horizontal minimal equation \eqref{eh2me}
($0\leqs\ep\leqs 1$) in order to solve  the Dirichlet Problem for $\ep>0$; then, by a compactness argument, we solve the original horizontal minimal equation.

The techniques use suitable barriers and the maximum principle, combined  with our  a priori uniform $C^1$ estimates. Then, we are able to apply
  the Leray-Schaulder degree  theory \cite{A} and method  \cite{GT}, \cite{GD}, to accomplish the proof.

  Consider the following Dirichlet problem in two independent variables.

\begin{multline}\label{dir}
\hfill \HMc_\ep (g) = 0 \quad\text{in}\quad \Omega\hfill\\
\hfill\quad\qquad\hphantom{0000000000000000000} g=f \quad  \text{on}\quad \Ga, \quad g\in C^{2,\, \alpha}(\bar{\Omega})\hfill\\
\end{multline}

Let $R(\Om,f)$ be the quantity defined in \defref{qua}. We have the following existence  result.

\begin{theorem}\label{exi}

Let $\Om\subset\partial_\infty \Hip^2\times\R$ be a $C^ {2, \a}$
bounded convex domain   for some $0<\a<1$. Let $\Ga=\partial \Om$
and let $f\in C^{2, \a}(\Ga)$ be a positive function.

Assume
further that $\dd R(\Om,f) \leqs\min\limits_\Ga f +\min\limits_\Ga
f\sqrt{\frac{\pi}{2}}.$

Then, for any $\ep\in [0,1]$, the Dirichlet problem \eqref{dir}
admits a  positive solution $g$.

Particularly, $f$ admits an  extension $g\in C^{2,\,\alpha}(\overl
{\Omega})$ satisfying the horizontal minimal  equation
$\eqref{2min}$ on $\Omega$.
 Furthermore, this solution is obtained as the limit in the $C^2$-topology of a sequence $g_{\ep_n}, 0<\ep_n<1$ satisfying \eqref{dir}, as $\ep_n\goto 0$.

\end{theorem}
\vskip2mm
We observe that, as in \thmref{gge}, the last assumption in \thmref{exi} is invariant by hyperbolic translations

\vskip2mm

\begin{proof}

 Suppose first $0<\ep\leqs 1.$ We first intend to show the existence of a  solution of the Dirichlet  problem \eqref{dir}, for $0<\ep\leqs 1.$

 Let $\ell(s)=2s -1,\, s\in [1/2, 1]$ and $\ell(s)=0,\, s\in
[0,1/2].$ Let $h(s,p)=(2s -1)f(p) + 2(1-s)\min\limits_\Ga f,\,
s\in [1/2, 1]$ and $h(s,p)=2s \min\limits_\Ga f,\, s\in [0, 1/2].$

 Given $s\in [0,1]$ and given $g\in
C^{1,\,\beta}(\overl{\Om}),$ consider the following family of
linear Dirichlet problems:
\begin{align*}
 & u_{xx} (g^2 +g_t^2+\ep) + u_{tt} (1 +g_x^2) -2g_xg_tu_{xt} +
\ell(s) \max\{g, 0\}(1 +g_x^2)=0\\
&\quad\text{in}\quad \Omega,\\
 &u(p)= h(s,p),\quad p\in \Ga.
\end{align*}
 Define an operator $T\,:C^{1,\,\beta}(\overl{\Om}) \times [0, 1] \rightarrow
C^{2,\,\alpha\beta}(\bar{\Omega})$, $u = T(g,\,s)$ by the unique
solution of the above problem for given $(g,\,s)$. Existence and
uniqueness of $u$ are ensured by the  extension lemma \cite[Lemma 6.38]{GT} and by the classical theorem for linear
strictly elliptic operator \cite[Theorem 6.14]{GT}. Note that $T(g, 0)=0$ and
$T(g, s)= 2s\min\limits_\Ga f,$ if $s\in [0, 1/2]$.

To ensure the existence of a solution of our Dirichlet problem in the space $C^{2,\,
\alpha}(\bar{\Omega})$  it is suffices to check that  $T(\cdot,1)$ has a fixed point.
The equation  $g= T(g, s)$ reads

\begin{multline}\label{esq}
  g_{xx} (g^2 +g_t^2+\ep) + g_{tt} (1 +g_x^2) -2g_xg_tg_{xt} +
\ell(s)\max\{g, 0\}(1 +g_x^2)=0\hfill\\
\qquad \mbox{in}\quad \Omega,\hfill \\
 \qquad g  = h(s,\cdot) \quad \mbox{on}\quad \Ga \hfill.
\end{multline}

Then,  we can apply maximum principle, comparing with Euclidean
planes $y=\cst$, to ensure that a solution  of the equation
\eqref{esq} is nonnegative, so  it satisfies

\begin{equation}\label{esq2}
\begin{split}
&  g_{xx} (g^2 +g_t^2+\ep) + g_{tt} (1 +g_x^2) -2g_xg_tg_{xt} +
\ell(s)g(1 +g_x^2)=0\\
&\mbox{in}\quad \Omega, \, \\
& g  = h(s,\cdot) \quad \mbox{on}\quad \Ga \;.
\end{split}
\end{equation}
with $g\geqs 0$ on $\Om.$

Noticing that by applying the maximum principle we have
 if $s\not=0,$ $g>0$ on $\Om$ and if $s=0,$ $g\equiv 0,$ {\em i.e.} $g$ satisfies equation
\eqref{2fme}, for $0<\ep \leqs 1.$ Of course, the definition of $\ell (s)$ ensures that  the solution $g$ for $0\leqs s\leqs 1/2,$ is constant equal to $h(s,p)=2s \min\limits_\Ga f.$   Observe that the solution $g$ for $s\geqs 1/2,$ satisfies the inequality $g\geqs\min\limits_\Ga f$ on $\Om.$
Now in virtue of our a priori global $C^1$ estimates (\thmref{ggc}) and the
global H\"older estimates of Ladyzhenskaya and Ural'tseva \cite{GT}, we have  a priori
global Hölder  estimates for the first derivatives. That
is, there exists a constant $C$ such that $[D
g]_{\Om,\beta}\leqs C,$ for all $g$ satisfying \eqref{esq2}. Hence, by employing
 the  Leray-Schauder theorem  \cite{A}, \cite{GT}, we obtain the desired existence of a positive solution $g_\ep$ of the Dirichlet problem
 \eqref{dir}, if $0<\ep\leqs 1.$ Recall that the uniform a priori horizontal length estimates, given by \lemref {hhen},  forces the  uniform lower bound $g_\ep> \min\limits_\Ga f $ on $\Om$, independently of $\ep$.

 Now let $\ep_n$ be a sequence such that $\ep_n\goto 0$, if $n\goto
 \infty$ ($0<\ep_n\leqs 1$) and let $g_{\ep_n}$ be a positive  $C^{2,\,
\alpha}(\overl{\Omega})$ solution of \eqref{dir}.  Our a priori global $C^1$ estimates  combined with the a priori Schauder global estimates, allow us to apply  the Arzelà-Ascoli's theorem to obtain  a subsequence
$\{g_{\ep_{n_j}}\}$ that converges to a $C^2(\overl{\Omega})$
nonnegative function $g$ satisfying \eqref{eh2me}. Clearly, $g\geqs  \min\limits_\Ga f $  on $\Om.$ Henceforth we have
 a solution of the Dirichlet problem \eqref{dir}, for $\ep=0.$
This accomplishes the proof of the theorem.
\end{proof}

\vskip2mm
\begin{remark}\label{cat} {\em
The geometry of the ambient space   $\hip^2 \times \R$ has some very intriguing geometric phenomenon: Any catenoid (minimal surface of revolution) has vertical height less than $\pi$ and the supremum of the family is $\pi$. Eric Toubiana and the author, using the family of catenoids as suitable barriers, proved an asymptotic principle \cite[Theorem 2.1]{SE-T} that have many consequences. In particular, it follows that there is
no horizontal minimal graph given by a function $g\in C^2(\Om)\cap
C^0(\overl{\Om})$ on a bounded strictly convex domain $\Om$,
taking zero  boundary data on $\partial \Om.$ We believe that the fact that the
 strict convexity of the Jordan domain  forbids the
existence of the Dirichlet problem for \eqref{2min}, {\em with
zero boundary data}, is a very surprising phenomenon.

 Some other results in the minimal surfaces theory, make use of the behavior of the catenoid family  \cite{NSWT}, \cite{H-N-SaE-T}.}

\end{remark}

On the other hand, the following existence result is somehow a counterpart of the above remark and is a immediate consequence of \thmref{exi}.

\newpage
\begin{corollary}\label{coex}
Let $\Om\subset\partial_\infty \Hip^2\times\R$ be a $C^ {2, \a}$
bounded convex domain   for some $0<\a<1$. Let $\Ga=\partial \Om$
and let $f\in C^{2, \a}(\Ga)$ be a positive function.

Let $c_0$ be a constant satisfying the inequality
 $\dd c_0\geqs \osc\limits_\Ga (f) +\frac{h(\Ga)}{2}\cdot$  Then, for any $\ep\in [0,1]$, there exists a solution of the  $\ep$-horizontal minimal equation taking the boundary data $f +c_0$ on $\Ga.$

\end{corollary}

\end{document}